\documentclass[12pt, reqno]{amsart}
\usepackage{amsmath, amsthm, amscd, amsfonts, amssymb, graphicx, color}
\usepackage[bookmarksnumbered, colorlinks, plainpages]{hyperref}

\input{mathrsfs.sty}

\newtheorem{theorem}{Theorem}[section]
\newtheorem{lemma}[theorem]{Lemma}

\theoremstyle{definition}

\theoremstyle{remark}

\begin{document}
\title{Some conditions implying normality of operators}
\author[M.S. Moslehian, S.M.S. Nabavi Sales]{M. S. Moslehian$^1$ and S. M. S. Nabavi Sales$^2$}
\address{$^1$ Department of Pure Mathematics, Center of Excellence in
Analysis on Algebraic Structures (CEAAS), Ferdowsi University of
Mashhad, P. O. Box 1159, Mashhad 91775, Iran.}
\email{moslehian@ferdowsi.um.ac.ir and moslehian@ams.org} \urladdr{\url{http://profsite.um.ac.ir/~moslehian/}}
\address{$^{2}$ Department of Pure Mathematics, Ferdowsi University of
Mashhad, P. O. Box 1159, Mashhad 91775, Iran;\newline Tusi
Mathematical Research Group (TMRG); Mashhad, Iran.}
\email{sadegh.nabavi@gmail.com}



\begin{abstract}
Let $T\in\mathbb{B}(\mathscr{H})$ and $T=U|T|$ be its polar
decomposition. We proved that (i) if $T$ is log-hyponormal
 or $p$-hyponormal and $U^n=U^\ast$ for some $n$, then $T$ is
 normal; (ii) if the spectrum of $U$ is contained in some open
 semicircle, then $T$ is normal if and only if so is its Aluthge transform $\widetilde{T}=|T|^{1\over2}U|T|^{1\over2}$.
\end{abstract}
\maketitle

\section{Introduction}

Let $\mathbb{B}(\mathscr{H})$ be the algebra of all bounded linear
operators on a complex Hilbert space $\mathscr{H}$ with the identity
$I$. A subspace $\mathscr{K}\subseteq \mathscr{H}$ is said to reduce
$T\in\mathbb{B}(\mathscr{H})$ if both $T\mathscr{K}\subseteq
\mathscr{K}$ and $T^\ast \mathscr{K}\subseteq \mathscr{K}$ hold. We say
that an operator $T$ is $p$-hyponormal for some $p>0$ if $(T^\ast
T)^p\geq (TT^\ast)^p$. If $p=1$, $T$ is said to be hyponormal. Clearly
$T$ is hyponormal if and only if $\|T \xi\|\geq\|T^\ast \xi\|$ for
any $\xi\in\mathscr{H}$. If $T$ is an invertible operator satisfying
$\log(T^\ast T)\geq \log(TT^\ast)$, then it is called log-hyponormal, see \cite{TAN}.\\
Let $T=U|T|$ be the polar decomposition of $T$, where
$\ker(U)=\ker(|T|)$ and $U^\ast U$ is the projection onto
$\overline{{\rm ran}(|T|)}$. It is known that if $T$ is invertible,
then $U$ is unitary and $|T|$ is also invertible. It is easy to see
that
\begin{eqnarray}\label{H1}
|T^\ast|^s=U|T|^sU^\ast\end{eqnarray} for every nonnegative number
$s$. If $T$ is invertible, then
\begin{eqnarray}\label{H2}\log|T^\ast|=U(\log |T|)U^\ast.\end{eqnarray}
The Aluthge transformation $\widetilde{T}$ of $T$ is defined by
$\widetilde{T}:=|T|^{1\over2}U|T|^{1\over2}$. This notion was first introduced by Aluthge
\cite{ALU} and is a powerful tool in the operator theory. The reader is
referred to \cite{FUR} for undefined notions and terminology.

One interesting problem in the operator theory is to investigate some conditions under
which certain operators are normal. Several mathematicians have
paid attention to this problem, see \cite{ALU, ALW,JTU,OKU} and
references therein. One of interesting articles, which presents
some results about this topic is that of Stampfli \cite{STA}. He showed, among other things,  that
for a hyponormal operator $A$, if $A^n$ is normal for some positive
integer $n$, then $A$ is normal. The problem had already been considered in the case when $n=2$ by Putnam \cite{PUT}. The
results were generalized later to the other classes of operators by a
number of authors, for instance, Embry \cite{EMB}, Radjavi and
Rosenthal \cite{RAR} and Duggal \cite{DUG}. There is another point
of view about this issue via spectrum ${\rm sp}(\cdot)$. In \cite{STA} it is proved
that if the spectrum of a hyponormal operator contains only a finite
number of limited points or has zero area, then the operator is
normal. Using Aluthge transform, this aspect is generalized to
$p$-hyponormal and $\log$-hyponormal operators. In fact, if $T$ is
$p$-hyponormal or $\log$-hyponormal, then $\widetilde{\widetilde{T}}$ is
hyponormal \cite[Theorem 1.3.4.1 and Theorem 2.3.4.2]{FUR}. Due to
${\rm sp}(A)={\rm sp}(\widetilde{\widetilde{A}})$ \cite[Corollary 2.3]{ALW},
$\widetilde{\widetilde{A}}$ is normal. Now the result is concluded from the
fact that $\widetilde{\widetilde{A}}$ is normal if and only if so is $A$
\cite[Lemma 3]{UCT}. There are some applications of the subject in other areas of the
operator theory that was a motivation for our work, see \cite{OKU}.

In this paper we present some new conditions under which cartain
operators are normal. We also use a Fuglede--Putnam commutativity
type theorem to show that an invertible operator $T=U|T|$, where ${\rm sp}(U)$ is contained in an open semicircle, is normal if and
only if so is $\widetilde{T}$.

\section{Main results}
We start this section with one of our main result.

\begin{theorem}\label{t2.1}
Let $T\in\mathbb{B}(\mathscr{H})$ be $\log$-hyponormal or
$p$-hyponormal  and $T=U|T|$ be the polar decomposition of $T$ such
that $U^{n_0}=U^\ast$ for some positive integer $n_0$. Then $T$ is
normal.
\end{theorem}
\begin{proof}
Assume that $T$ is $p$-hyponormal for some $p>0$. Hence
$|T|^{2p}\geq|T^\ast|^{2p}=U|T|^{2p}U^\ast$ by \eqref{H1}. By
multiplying both sides of this inequality  by $U$ and $U^\ast$ we
have $U|T|^{2p}U^\ast\geq U^2|T|^{2p}U^{2\ast}$ whence $|T|^{2p}\geq
U|T|^{2p}U^\ast\geq U^2|T|^{2p}U^{2\ast}$. By repeating this
process, we reach the following sequence of operator inequalities:
\begin{eqnarray}\label{2}
|T|^{2p}\geq
 |T^\ast|^{2p}=U|T|^{2p}U^\ast\geq U^2|T|^{2p}U^{2\ast}\geq\ldots\geq
U^{n_0+1}|T|^{2p}U^{(n_0+1)\ast}\geq \ldots.
\end{eqnarray}
Because of $U^{n_0}=U^\ast$ we can observe that $U^{n_0+1}=U^\ast
U=U^{( n_0+1)\ast}$ is the projection onto $\overline{ran(|T|)}$.
Hence $U^{n_0+1}|T|^{2p}U^{(n_0+1)\ast}=|T|^{2p}$, from which and
inequalities \eqref{2} we obtain $|T|^{2p}=|T^\ast|^{2p}$. Hence $|T|^2=|T^\ast|^2$, i.e., $T$ is normal as desired.

In the case that $T$ is a $\log$-hyponormal operator
inequalities \eqref{2} are replaced by the inequalities
\begin{eqnarray*}
\log|T|\geq
 \log|T^\ast|=U(\log|T|)U^\ast&\geq& U^2(\log|T|)U^{2\ast}\nonumber\\
 &\geq& \ldots\geq
U^{n_0+1}(\log|T|)U^{(n_0+1)\ast}\geq \ldots
\end{eqnarray*} and the rest of the
proof is similar to argument above.
\end{proof}

We will need the following lemma in the sequel. One can easily prove it by using the fact that $\log(cT)=(\log c)I+\log T$.

\begin{lemma}\label{mos3}
If $T$ and $S$ are two invertible positive operators such that $\log
T\geq \log S$ and $c$ is a positive number, then $\log (cT)\geq \log
(cS)$.
\end{lemma}
\begin{theorem}
Let $T\in\mathbb{B}(\mathscr{H})$ be $\log$-hyponormal or
$p$-hyponormal  and $T=U|T|$ be the polar decomposition of $T$ such
that $U^{\ast n}\to I$ or $U^n\to I$ as $n \to
\infty$, where limits are taken in the strong operator topology. Then $T$ is normal.
\end{theorem}
\begin{proof} We assume that $U^{\ast n}\xi\to \xi$ as $n \to
\infty$ for all $\xi\in\mathscr{H}$. In the case $U^n\to I$ in the strong operator topology
a similar argument can be used. Let $T$ be $p$-hyponormal and
$\xi\in\mathscr{H}$. It follows from \eqref{2} that
\begin{eqnarray}\label{3}
\|\,|T|^{p}\xi\|\geq
\|\,|T^\ast|^{p}\xi\|=\|\,|T|^{p}U^\ast\xi\|\geq
\|\,|T|^{p}U^{2\ast}\xi\|\geq\ldots\geq
\|\,|T|^{p}U^{n\ast}\xi\|\geq \ldots.
\end{eqnarray}
Since
$$|\,\|\,|T|^pU^{\ast
n}\xi\|-\|\,|T|^p\xi\|\,|\leq\|\,|T|^pU^{\ast
n}\xi-|T|^p\xi\|\leq\|\,|T|^p\|\,\|U^{\ast n}\xi-\xi\|\to 0$$ as $n
\to \infty$, we have $\|\,|T|^pU^{\ast n}\xi\|\to\|\,|T|^p\xi\|$ as
$n \to \infty$. Hence, by \eqref{3} we get $\|\,|T|^{p}\xi\|^2=
\|\,|T^\ast|^{p}\xi\|^2$, so $|T|^{2p}= |T^\ast|^{2p}$. Thus $T$ is
normal.

Now let $T$ be a $\log$-hyponormal operator. Since $T$ is invertible
there exists  $c>0$ such that $c|T^\ast|\geq I$, so $\log
(c|T^\ast|)\geq 0$. Due to $\log|T|\geq
\log|T^\ast|=U(\log|T|)U^\ast$ we have  $\log(c|T|)\geq
\log(c|T^\ast|)=U\log(c|T|)U^\ast$ by Lemma
\ref{mos3} and equality \eqref{H2}. The rest of the proof is similar to
 the argument above and the proof of Theorem \ref{t2.1} so we omit it.

\end{proof} In
the sequel we are going to present a relationship between an
operator and its Aluthge transform. We essentially apply the
following lemma.

\begin{lemma}\cite{TAK}
Let $T,S\in\mathbb{B}(\mathscr{H})$. Then the following assertions
are equivalent:

({\rm i}) If $TX=XS$, then $T^\ast X=XS^\ast$ for any
$X\in\mathbb{B}(\mathscr{H})$.

({\rm ii}) If $TX=XS$ where $X\in\mathbb{B}(\mathscr{H})$, then
$\overline{R(X)}$ reduces $T$, $(\ker X)^\bot$ reduces $S$, and
operators $T|_{\overline{R(X)}}\,\,$ and $S|_{(\ker X)^\bot}$ are
normal.
\end{lemma}

\begin{theorem}
Let $T\in\mathbb{B}(\mathscr{H})$ be an invertible operator and
$T=U|T|$ be the polar decomposition of $T$. Let ${\rm sp}(U)$ be
contained in some open semicircle. Then $\widetilde{T}$ is normal if and
only if so is $T$.
\end{theorem}
\begin{proof} Assume that $\widetilde{T}$ is normal. Hence $\widetilde{T}X=X\widetilde{T}$ implies $\widetilde{T}^\ast
X=X\widetilde{T}^\ast$ for any $X\in\mathbb{B}(\mathscr{H})$ by
Fuglede--Putnam commutativity theorem. We first show that $TX=XT$
implies $T^\ast X=XT^\ast$ for any $X\in\mathbb{B}(\mathscr{H})$.
Let $X\in \mathbb{B}(\mathscr{H})$ and $TX=XT$. Then $U|T|X=XU|T|$,
whence
\begin{eqnarray}\label{6}
\widetilde{T}(|T|^{1\over2}X|T|^{-1\over2})&=&
|T|^{1\over2}(U|T|^{1\over2}|T|^{1\over2}X)|T|^{-1\over2}\nonumber\\&=&
|T|^{1\over2}(X|T|^{-1\over2}|T|^{1\over2}U|T|)|T|^{-1\over2}\nonumber\\
&=&(|T|^{1\over2}X|T|^{-1\over2})\widetilde{T}.
\end{eqnarray}
By \eqref{6} and the assumption  with $|T|^{1\over2}X|T|^{-1\over2}$
instead of $X$ we have
\begin{eqnarray*}
|T|^{1\over2}U^\ast|T|X|T|^{-1\over2}&=&|T|^{1\over2}U^\ast|T|^{1\over2}(|T|^{1\over2}X|T|^{-1\over2})=
\widetilde{T}^\ast(|T|^{1\over2}X|T|^{-1\over2})\\&=&(|T|^{1\over2}X|T|^{-1\over2})\widetilde{T}^\ast
=|T|^{1\over2}X|T|^{-1\over2}|T|^{1\over2}U^\ast|T|^{1\over2}\\&=&|T|^{1\over2}XU^\ast|T|^{1\over2}.
\end{eqnarray*}
So that
\begin{eqnarray}\label{8}
U^\ast|T|X&=&XU^\ast|T|
\end{eqnarray}
and $|T|X|T|^{-1}=UXU^\ast$. Therefore
$$|T|X|T|^{-1}=U^\ast(U|T|X)|T|^{-1}=U^\ast(XU|T|)|T|^{-1}=U^\ast
XU\,.$$
Thus $UXU^\ast=U^\ast XU$,  whence $U^2X=XU^2$.

Now we use Beck and Putnam argument in \cite{BEP}. We replace $U$ by $e^\alpha U$ if
it is necessary  and assume that ${\rm sp}(U)$ is contained in the
set $\{e^{i\lambda}: \varepsilon<\lambda<\pi-\varepsilon\}$ for some
$\varepsilon>0$. Let
$U=\int_\varepsilon^{\pi-\varepsilon}e^{i\lambda} dE(\lambda)$ be
the spectral decomposition of $U$. One has
$U^2=\int_{2\varepsilon}^{2\pi-2\varepsilon}e^{i\lambda}
dF(\lambda)$, where $F(\lambda)=E({\lambda\over2})$. By $U^2X=XU^2$
we have $U^{2n}X=XU^{2n}$ for every $n\in\mathbb{Z}$,
 so $U^{2n}=\int_{2\varepsilon}^{2\pi-2\varepsilon}e^{in\lambda}
 dF(\lambda)$. Hence
 $f(U^2)X=Xf(U^2)$ for every $f$ in the
set of all bounded Borel-measurable complex-valued functions on
$\{z:|z|=1\}$ since $\{e^{int}\}$ is complete on the interval $0\leq
t\leq2\pi$. Hence, by spectral resolution for normal operator $U$,
$F(\lambda)X=XF(\lambda)$, whence $E(\lambda)X=XE(\lambda)$ and this
implies again that $UX=XU$ and clearly this implies that
\begin{eqnarray}\label{13}
U^\ast X&=&XU^\ast.
\end{eqnarray}
From \eqref{8} and \eqref{13} we obtain
\begin{eqnarray}\label{14}
|T|X=U(U^\ast|T|X)&=&U(XU^\ast|T|)=U(U^\ast X)|T|=X|T|.
\end{eqnarray}
From \eqref{13} and \eqref{14} we deduce that $T^\ast X=|T|U^\ast
X=X|T|U^\ast=XT^\ast$ as desired. We have shown that $TX=XT$ implies
$T^\ast X=XT^\ast$ for any $X\in\mathbb{B}(\mathscr{H})$. It follows
from Lemma 2.4({\rm ii}) for $X=I$  that $T$ is normal.

The reverse is easy. In fact if $T$ is normal, then $\widetilde{T}=T$.
\end{proof}

\textbf{Acknowledgement.} The authors would like to sincerely thank
the referee for very useful comments and suggestions.

\end{document}